\documentclass[11pt]{amsart}
\usepackage{amsfonts,amssymb,amscd,amsmath,enumerate,verbatim,calc,amsthm, mathrsfs}
\usepackage[margin=1in]{geometry}
\usepackage{paralist}
\usepackage[colorlinks,citecolor=blue,linkcolor=blue]{hyperref}
\usepackage[mathscr]{euscript} 
\usepackage{color}
\usepackage{graphicx}
\usepackage[all]{xy}
\usepackage{tikz-cd}
\usepackage[english]{babel}

\def\Proj{\operatorname{Proj}}  
\def\Spec{\operatorname{Spec}}
\def\Hom{\operatorname{Hom}}
\def\Min{\operatorname{Min}}

\def\rdim{\operatorname{rdim}}

\def\sym{\operatorname{Sym}}

\def\l{{\lambda}}

\def\m{{\mathfrak{m}}}
\def\n{{\mathfrak{n}}}

\def\thebibliography#1{\section*{References}
\list{[\arabic{enumi}]}{\settowidth \labelwidth{[#1]} \leftmargin
\labelwidth \advance \leftmargin \labelsep \usecounter{enumi}}
\def\newblock{\hskip .11em plus .33em minus .07em} \sloppy
\clubpenalty 4000 \widowpenalty 4000 \sfcode`\.=1000 \relax}

\def\frk{\mathfrak}

\def\Phi{{\frk N}}
\def\Z{{\mathbb Z}}
\def\FF{{\mathcal F}}

\def\opn#1#2{\def#1{\operatorname{#2}}} 
\opn\chara{char} \opn\length{\ell} \opn\pd{pd} \opn\rk{rk}
\opn\projdim{proj\,dim} \opn\injdim{inj\,dim}
\opn\rank{rank} \opn\depth{depth} \opn\grade{grade} 
\opn\hei{ht} \opn\embdim{emb\,dim}\opn\codim{codim}
\opn\Tr{Tr} \opn\bigrank{big\,rank}
\opn\gr{\gr}
\opn\superheight{superheight} \opn\lcm{lcm}
\opn\rdim{rdim} \opn\trdeg{tr\,deg} \opn\reg{reg}  \opn\lreg{lreg} 
\opn\ini{in} \opn\lpd{lpd} \opn\size{size} \opn{\mult}{mult}

\opn\div{div} \opn\Div{Div} \opn\cl{cl} \opn\Cl{Cl}
\opn\Spec{Spec} \opn\Supp{Supp} \opn\supp{supp} 
\opn\Sing{Sing} \opn\Ass{Ass} \opn\Min{Min}
\opn\Proj{Proj}
\opn\Ann{Ann} \opn\Rad{Rad} \opn\Soc{Soc} \opn\Supp{Supp}
\opn\Syz{Syz} \opn\Im{Im} \opn\Ker{Ker} \opn\Coker{Coker}
\opn\Am{Am} \opn\Hom{Hom} \opn\tor{Tor} \opn\Ext{Ext}
\opn\End{End} \opn\Aut{Aut} \opn\id{id}

\opn\nat{nat} \opn\pff{pf} 
\opn\Pf{Pf} \opn\GL{GL} \opn\SL{SL} \opn\mod{mod} \opn\ord{ord}
\opn\Gin{Gin} \opn\Hilb{Hilb}
\opn\adeg{adeg} \opn\std{std}\opn\ip{infpt}
\opn\Pol{Pol} \opn\sat{sat} \opn\Var{Var}
\opn\aff{aff} \opn\con{conv} \opn\relint{relint} \opn\st{st}
\opn\lk{lk} \opn\cn{cn} \opn\core{core} \opn\vol{vol}
\opn\link{link} \opn\star{star}
\opn\gr{gr}

\def\pot#1#2{#1[\kern-0.28ex[#2]\kern-0.28ex]}

\opn\dirlim{\underrightarrow{\lim}}
\opn\inivlim{\underleftarrow{\lim}}

\newtheorem{Theorem}{Theorem}[section]
\newtheorem{Lemma}[Theorem]{Lemma}

\newtheorem{Proposition}[Theorem]{Proposition}
\newtheorem{Remark}[Theorem]{Remark}

\newtheorem{Definition}[Theorem]{Definition}

\let\phi=\varphi
\let\kappa=\varkappa

\opn\dis{dis}
\opn\Lex{Lex}

\opn{\MinAss}{MinAss}

\title[]{Tight closure of powers of parameter ideals in hypersurface rings and their tight Hilbert polynomials}

\author{{Saipriya Dubey}, {Vivek Mukundan}
\and{Jugal  Verma}}
\thanks{The first author is supported by the Senior Research Fellowship of HRDG, CSIR, Government of India}
\address{Saipriya Dubey \newline\indent Mathematics Department, Indian Institute of Technology Bombay, \newline \indent 
Powai, Mumbai - 400076, India, Email- {\normalfont{sdubey@math.iitb.ac.in}} \vspace{0.28cm}
\newline\indent  Vivek Mukundan \newline\indent Mathematics Department, Indian Institute of Technology Delhi, \newline \indent 
Hauz Khas, Delhi - 110016, India, Email- {\normalfont {vmukunda@iitd.ac.in }}
\vspace{0.28cm}
\newline\indent  Jugal Verma \newline\indent Mathematics Department, Indian Institute of Technology Bombay, \newline \indent 
Powai, Mumbai - 400076, India, Email- {\normalfont {jkv@iitb.ac.in}}}
\dedicatory{Dedicated to Professor J\"urgen Herzog on the occasion of his 80th birthday}

\begin{document}
\maketitle
\begin{abstract}
In this paper we find the tight closure of powers of parameter ideals of certain diagonal hypersurface rings. In many cases the associated graded ring with respect to tight closure filtration turns out to be Cohen-Macaulay. This helps us find the tight Hilbert polynomial in these diagonal hypersurfaces. We determine the tight Hilbert polynomial in the following cases:
{\rm (1)} $F$-pure diagonal hypersurfaces where number of variables is equal to the degree of defining equation,
{\rm (2)} diagonal hypersurface rings where characteristic of the ring is one less than the degree of defining equation and
{\rm (3)} quartic diagonal hypersurface in four variables.
\end{abstract}
\section{Introduction}  
The objective of this paper is to determine the tight Hilbert polynomial of ideals generated by homogeneous system of parameters in diagonal hypersurface rings in prime characteristic. The tight Hilbert polynomial was introduced by K. Goel, V. Mukundan and J. K. Verma in \cite{GMV2019}. Let $(R,\m)$ be a $d$-dimensional analytically unramified local ring of prime characteristic $p>0$ and $I$ be an $\m$-primary ideal. Let $I^*$ denote the tight closure of $I$ (refer Definition \ref{def_tytcl,testele}). Let $H_I^*(n)=\ell \left(\frac{R}{(I^n)^*} \right)$ where $\ell$ denotes the length. Then for large $n,$ $H_I^*(n)$ is given by a polynomial $P_I^*(n)$ written as $$P_I^*(n)=e_0^*(I)\binom{n+d-1}{d}-e_1^*(I)\binom{n+d-2}{d-1}+\cdots+(-1)^de_d^*(I),$$ where $e_i^*(I)\in \mathbb{Z}.$ Recall that $R$ is called an $F$-rational local ring if the ideals of the principal class, that is, ideals $J$ generated by height $J$ elements, are tightly closed. It is proved \cite{GMV2019} that in a Cohen-Macaulay local ring $R$, $e_1^*(I)=0$ if and only if $R$ is $F$-rational. This indicates that the other tight Hilbert coefficients might play an important role in detecting other $F$-singularities. They also computed the tight Hilbert polynomial for two-dimensional diagonal hypersurfaces $K[X, Y, Z]/(X^r+Y^r+Z^r),$ where $K$ is a field of prime characteristic $p$ and $(p,r)=1.$  Since there are only a handful of examples of tight Hilbert polynomials known in the literature, this paper is devoted to computation of the tight Hilbert polynomial of parameter ideals $I=(x_1,\ldots ,x_d)$ of diagonal hypersurfaces in higher dimensions, i.e. $R=\mathbb{F}_p[X_1,X_2,\ldots,X_{d+1}]/(X_1^r+X_2^r+\cdots+X_{d+1}^r)$ with the unique homogeneous maximal ideal $\m=(I,x_{d+1})$ in some special cases. (Throughout, the lower case letters denote the images of the corresponding variables.)

The tight Hilbert coefficients $e_i^*(I)$ can be computed via the tight closure of the powers of the paramter ideals $(I^n)^*$. In general, computing the tight closure of the powers can be quite hard. In a hypersurface ring, the tight closure of a homogeneous system of parameters can be computed using the \textit{strong vanishing theorem }\cite[Theorem 6.4]{Huneke1998}. This result computes the tight closure of the homogeneous system of parameters when the characteristic is ``large enough'' compared to the degree of the defining ideal of the hypersurface. It is interesting to note that the tight closure of $I$ equals $I+\m^{d}$ in this case. Much of this work has been a result of searching for ideals $I$ whose tight closure of the powers $(I^n)^*$ has a similar description, such as $I^n+\m^k$ for some $k$, as that appearing in the strong vanishing theorem.

This paper is organized as follows. Section \ref{sec_preli} contains basic definitions, notations and results which will provide all the tools necessary in the later sections. In Section \ref{sec_HPfilt}, we find the Hilbert polynomial of the filtration $\{I_n\}_{n\geq 0},$ where $I_0=R$ and $I_n=I^n+\m^{n+t}$ for $n\geq 1$ that occur (in the subsequent sections) as the tight closure filtration of ideals in diagonal hypersurface rings . 
In Section \ref{sec_THP_diag_Hy}, we determine the tight Hilbert polynomial in the ring $R=\mathbb{F}_p[X_1,X_2,\ldots,X_{d+1}]/(f),$ where degree of $f$ is $d+1$ and the characteristic $p>(d-1)(d+1)-d.$ Here, the associated graded ring of the tight closure filtration $\{(I^n)^*\}_{n\geq 0}$ is given by $\gr_{I}^*(R)=\bigoplus\limits_{n\geq 0}\frac{(I^n)^*}{(I^{n+1})^*}.$
\begin{Theorem}
Let $T=\mathbb{F}_p[X_1,X_2,\ldots,X_{d+1}],$ $R=T/(f)$ where $f\in T$ is a square-free, monic in $X_{d+1}$ homogeneous polynomial of degree $d+1$ and $p>(d-1)(d+1)-d.$ Let $I=(x_1,\ldots ,x_d)$ and $\m=(I,x_{d+1}).$ Then $(I^n)^*=I^n+\m^{n+d-1}$ and for $n\geq 1,$
$$\ell \left( \frac{R}{(I^n)^*}\right)=(d+1)\binom{n+d-1}{d}-\binom{n+d-2}{d-1}.$$ Moreover $\gr ^*_I(R)$ is Cohen-Macaulay.
\end{Theorem}
One of the important ingredient of the proof is the use of \cite{kod_van} which gives bounds on the tight closure of $I$ using the betti numbers appearing in a resolution of $R/I$. Here we use the Eagon-Northcott complex which resolves $R/I^n$ as a means of constructing the tight closure of the powers $I^n$.
 
Recall that, a Noetherian local ring $R$ of prime characteristic $p$ is said to be $F$-pure if the Frobenius homomorphism $F:M\rightarrow M\otimes_R F(R)$ is injective for all $R$-modules $M$ where $F(R)$ is the image of $R$ under the frobenius endomorphism which takes $r\in R$ to $r^p.$ We know that the test elements play a vital role in the theory of tight closure. In $F$-pure rings, the ``test ideal" turns out to be a radical ideal { \cite[Proposition 2.5]{Fed_Wat}}. Exploiting this property of $F$-pure rings, in Theorem \ref{THP_Fpure} we determine the tight Hilbert polynomial in all diagonal hypersurface $F$-pure rings when the number of variables is equal to the degree of the defining equation of the hypersurface. Along with results on tight Hilbert polynomial of a minimal reduction of the maximal homogeneous ideals, we also present the result on F-pure rings.
\begin{Theorem}
Let $R=\mathbb{F}_p[X_1,X_2,\ldots,X_{t+2}]/(X_1^{r}+X_2^{r}+\cdots +X_{t+2}^{r}),$ $I=(x_1,\ldots ,x_{t+1})$ and $\m=(I,x_{t+2}),$ where $p>tr-(t+1).$ Then\\
{\rm (1)} $\gr_I^*(R)$ is Cohen-Macaulay.\\
{\rm (2)} The tight Hilbert polynomial $P_I^*(n)$ is given by  \[
    P_I^*(n)= 
\begin{cases}
    r\binom{n+t}{t+1},& \text{ if } r\leq t+1,\\
    r\binom{n+t}{t+1}-\binom{r-t}{2}\binom{n+t-1}{t}+\binom{r-t}{3} \binom{n+t-2}{t-1}+\cdots +(-1)^j\binom{r-t}{j+1}\binom{n+t-j}{t+1-j},& \text{ if }r\geq t+2.\\
\end{cases}
\]
If $r\geq t+2,$ $e_j^*(I)=0$ for $r-t\leq j \leq t+1.$\\
{\rm (3)} If $r=t+2$ and $R$ is F-pure, then for all $n\geq 1$,
$$\ell \left( \frac{R}{(I^n)^*}\right)=(t+2)\binom{n+t}{t+1}-\binom{n+t-1}{t}.$$ 
\end{Theorem}

In Section \ref{sec_THP p+1 degree}, we explore the tight Hilbert polynomial in the diagonal hypersurface rings where the characteristic of the ring is one less than the degree of the defining equation.
\begin{Theorem}
Let $R=\frac{\mathbb{F}_p[X_1,X_2,\ldots , X_{d+1}]}{(X_1^{p+1}+X_2^{p+1}+\cdots +X_{d+1}^{p+1})},$ $I=(x_1,x_2,\ldots ,x_d)$ and $\m=(I, x_{d+1}).$ Then\\
\rm {(1)}  $(I^n)^*=I^n+\m ^{n+1}$ for all $n\geq 1.$ \\
\rm {(2)} $P_I^*(n)=(p+1)\binom{n+d-1}{d}-\binom{p}{2}\binom{n+d-2}{d-1}+\cdots +(-1)^j\binom{p}{j+1}\binom{n+d-j-1}{d-j},$ where $j=\min \{p-1,d\}.$\\ Moreover, $e_i^*(I)=0$ for $p\leq i \leq d$ and $\gr ^*_I(R)$ is Cohen-Macaulay.
\end{Theorem}
 Finally in Section \ref{sec_QuarticTHP}, we determine the tight Hilbert polynomial for quartic diagonal hypersurfaces in four variables.
\begin{Theorem}
Let $R=\frac{\mathbb{F}_p [X,Y,Z,W]}{(X^4+Y^4+Z^4+W^4)},$ $I=(x,y,z)$ and $\m=(I,w).$ Then for $n\geq 1,$
\[
    P_I^*(n)= 
\begin{cases}
    4\binom{n+2}{3}-3\binom{n+1}{2}+n,& \text{ if } p=3,\\
    4\binom{n+2}{3}-\binom{n+1}{2},& \text{ if } p\geq 5.
\end{cases}
\]
Moreover $\gr ^*_I(R)$ is Cohen-Macaulay.
\end{Theorem} 
\subsection*{Acknowledgements}We thank Anurag Singh and Karen Smith for discussions about computation of tight closure of powers of ideals.
\section{Preliminaries}\label{sec_preli}
We first set up notation, recall certain basic concepts and results needed in subsequent sections.
\begin{Definition}
Let $(R,\m)$ be a local ring and $I$ be an $\m$-primary ideal of $R.$  A sequence of ideals $\FF= \lbrace I_n \rbrace _{n\in \mathbb{Z}}$ 
is called an {\bf $I$-filtration} if for all $m,n\in \mathbb{Z},$  

\centerline{{\rm (1)} $I_n=R$ for all $n\leq 0,\;\;$  {\rm (2)}  $I_n\supseteq I_{n+1},\;\;$ {\rm (3)}  $I_mI_n\subseteq I_{m+n}\;\; $ and {\rm (4)} $ I^n\subseteq I_n.$ } 
\noindent The $I$-filtration $\FF$ is called {\bf $I$-admissible}  if there exists $r\in \mathbb{N}$ such that $I_{n}\subseteq I^{n-r} $ for all $n\in \mathbb{Z}.$
\end{Definition}

We associate three blowup algebras with an $I$-admissible filtration $\FF=\lbrace I_n\rbrace_{n\geq 0}$. The Rees algebra of $\FF$ is denoted by $\mathcal{R}(\FF)=\bigoplus_{n \geq 0} I_n t^n$, the extended Rees algebra of $\FF$ is denoted by $\mathcal{R}'(\FF)=\bigoplus_{n \in \Z} I_n t^n$ and the associated graded ring of $\FF$ is denoted by $\gr_{\FF}(R)=\bigoplus_{n\geq 0} I_n/I_{n+1}.$ If $\FF=\lbrace I^n\rbrace_{n\geq 0},$ we denote the blow up algebras by $\mathcal{R}(I), \mathcal{R}'(I)$ and $\gr_I(R)$ respectively.
Let $I$ be an $\m$-primary ideal. The\textbf{ Hilbert function of $\FF$} is defined by $H_{\FF}(n)=\ell(R/I_n).$ If $\FF$ is an $I$-admissible filtration, then for sufficiently large $n,$ $H_{\FF}(n)$ coincides with a polynomial $P_{\FF}(n)$ of degree $d$ for large $n,$ where $d=\dim R.$ There exist integers $e_0(\FF),e_1(\FF),\ldots,e_d(\FF)$ so that
$$P_{\FF}(n)=e_0(\FF)\binom{n+d-1}{d}-e_1(\FF)\binom{n+d-2}{d-1}+ \cdots +(-1)^de_d(\FF)$$ and it is called the \textbf{Hilbert polynomial of $\FF.$} The uniquely determined integers $e_i(\FF)$ are called the \textbf{Hilbert coefficients of $\FF.$} The coefficient $e_0(\FF)$ is a positive integer and it is called the \textbf{multiplicity of $\FF$}. The coefficient $e_1(\FF)$ is called as the \textbf{Chern number of $\FF.$} 
For $\FF=\lbrace I^n \rbrace_{n\in \mathbb{N}},$  $e(I):=e_0(\FF)=e_0(I)$ (resp. $e_1(I):=e_1(\FF)$) is called the \textbf{multiplicity} (resp. the \textbf{Chern number}) of $I.$ 
The \textbf{reduction number} of $\mathcal{F}$ with respect to a minimal reduction $J\subseteq I$ is defined as $r_J(\mathcal{F})=\min \{m|JI_n=I_{n+1}\text{ for all }n\geq m\}$ and the reduction number of $\mathcal{F}$ is $$r(\mathcal{F})=\min \{r_J(\mathcal{F})|J\text{ is a minimal reduction of } \mathcal{F}\}.$$ The \textbf{postulation number} of $\mathcal{F}$ is defined as $$\eta(\mathcal{F})=\min \{n|P_{\mathcal{F}}(m)=H_{\mathcal{F}}(m) \text{ for all }m>n\}.$$
\begin{Definition}
Let $R$ be a commutative ring and $I$ be an ideal of $R.$ An element $x\in R$ is {\bf integral over an ideal $I$}  if 
$$x^n+a_1x^{n-1}+\dots+a_n=0$$
where $a_i\in I^i$ for $i=1,2,\dots,n.$

The { \bf integral closure $\overline{I}$}  of $I$ is the ideal of all $x\in R$ which are integral over $I.$
 
 A Noetherian local ring $(R,\m)$ is said to be \textbf{analytically unramified} if its  completion with respect to the $\m$-adic topology is reduced.
\end{Definition}

D. Rees characterised \cite[Theorem 1.4]{reesau} analytically unramified local rings in terms of the normal filtration
$\{\overline{I^n}\}$ of any $\m$-primary ideal $I.$ 
\begin{Theorem} [\bf D. Rees]
A Noetherian local ring $(R,\m)$ is analytically unramified if and only if the
normal filtration $\{\overline{I^n}\}$ of any $\m$-primary ideal $I$ is $I$-admissible.
\end{Theorem}

Let the normal filtration of $I$ be given by $\mathcal{N}=\lbrace\overline{I^n}\rbrace_{n\in \mathbb{Z}}.$ Rees's Theorem implies that if $R$ is analytically unramified then ${\mathcal{R}}'(\mathcal{N})=\bigoplus_{n\in \mathbb{Z}}\overline{I^n}t^n$ is a finite module over $\mathcal{R}'(I).$ Therefore the normal Hilbert function $\overline{H}_I(n)=\ell(R/\overline{I^n})$ coincides with a  polynomial $\overline{P}_I(n)$ of degree $d$ for large $n,$ and is known as the \textbf{normal Hilbert polynomial of $I$.} We write
$$\overline{P} _I(n)=\overline{e}_0(I)\binom{n+d-1}{d}-\overline{e}_1(I)\binom{n+d-2}{d-1}+\cdots+(-1)^d\overline{e}_d(I).$$
Here $\overline{e}_0(I)=e(I),$ the multiplicity of $I$,  $\overline{e}_1(I)$ is known as the \textbf{normal Chern number of $I$} and the coefficients $\overline{e}_i(I)$ for $i=0,1,\ldots,d$ are called the \textbf{normal Hilbert coefficients of $I.$} In this case, we denote the blow up algebras $\mathcal{R}(\mathcal{N}),$ $\mathcal{R}'(\mathcal{N})$ and $\gr_{\mathcal{N}}(R)$ by $\overline{\mathcal{R}}(I),$ $\overline{\mathcal{R}'}(I)$ and $\overline{\gr}_I(R)$ respectively.

We now introduce the filtration $\{(I^n)^*\}$ of tight closure of powers an ideal $I.$
\begin{Definition}
For a ring $R,$ the subset of $R$ consisting of all the elements which are not contained in any minimal prime ideals of $R,$ is denoted by $R^\circ.$  
\end{Definition}

Let $R$ be a ring of prime characteristic $p.$ Let $q=p^e$ for $e\in \mathbb{N}.$ If $I=(a_1,\ldots,a_r)$ is an ideal of $R$ then  the $q^{th}$ Frobenius power of $I$ is the ideal $I^{[q]}=(a_1^q,\ldots,a_r^q).$ Notice that for any positive integer $n,$ $(I^n)^{[q]}=(I^{[q]})^n.$

\begin{Definition}\label{def_tytcl,testele}
An element $x\in R$ is said to be in the \textbf{tight closure $I^*$} of $I$ if there exists $c\in R^\circ$ such that $cx^q\in I^{[q]},$ for all sufficiently large $q=p^e.$ An element $c\in R^\circ$ is called a \textbf{test element} for $R$ if for all ideals $I$ in $R$ and $x\in I^*,$ $cx^q\in I^{[q]}$ for all $q=p^e.$\\
The \textbf{test ideal} of a ring $R,$ denoted by $\tau (R)$ is the ideal generated by the test elements of $R.$
\end{Definition}
{It is well-known that every reduced excellent local ring has a test element, see \cite[Theorem 6.1(a)]{Hoc_Hun_1994}. Taking a step forward in this direction, M. Hochster proved a result which gives test elements arising from every non-zero element of ``relative Jacobian ideal" \cite{Hoc_2004}. It is relevant for our setup as shown below.
\begin{Theorem}\cite[Corollary 8.2]{Hoc_2004}\label{Thm:Hoc_testele} Let $R$ be a domain that is module-finite over a regular domain $A$ of prime characteristic $p$ such that the extension of fraction fields is separable, then every non-zero element of $\mathcal{J}$ is a test element, where $\mathcal{J}=J(R/A)$ is the relative Jacobian ideal.
\end{Theorem}
Let $R=\frac{\mathbb{F}_p[X_1,\ldots ,X_{d+1}]}{(X_1^r+\cdots +X_{d+1}^r)}$ with $p\nmid r,$ and $x_1,\ldots,x_{d+1}$ be images of $X_1,\ldots,X_{d+1}$ respectively in $R.$ Then the extension $A=\mathbb{F}_p[x_1,\ldots,x_d] \subseteq \mathbb{F}_p[x_1,\ldots,x_{d+1}]=R$ satisfies all the assumptions of the above theorem, $x_{d+1}^{r-1}\in \mathcal{J}(R/A)$ is a test element of $R.$ By symmetry, $x_i^{r-1}\in \tau (R)$ for $1\leq i \leq d+1.$\\
Fedder and Watanabe proved the following result for reduced rings which concludes that for F-pure rings, $\tau(R)$ is radical ideal.
\begin{Proposition}\cite[Proposition 2.5]{Fed_Wat}\label{Prop_FedWat_test}
Let $R$ be a reduced ring. Let $I$ be an ideal of $R$ such that $I^{[q]}$ is contracted with respect to the Frobenius map for every $q=p^e.$ Then for any $x\in I^*,$ $J(x)=\cap_{q=p^e}(I^{[q]}:x^q)$ is a radical ideal.
\end{Proposition}
If $R$ is F-pure then it follows from the definition that every ideal is contracted with respect to the Frobenius map for every $q=p^e.$ Observe that
$$\tau(R)=\cap_{I\subset R}\left( \cap_{x \in I^*}(J(x)) \right),$$ where $J(x)$ is same as defined above. This implies that test ideal, $\tau(R)$ is radical ideal.\\}
The tight filtration of $I$ is given by $\mathcal{T}=\lbrace(I^n)^*\rbrace_{n\in \mathbb{Z}}.$ As  $(I^n)^*\subset \overline{I^n},$ for all $n\geq 1,$  $$\mathcal{R'}(I)\hookrightarrow \mathcal{R'}(\mathcal{T})\hookrightarrow \mathcal{R'}(\mathcal{N}).$$ If  $R$ is analytically unramified, $\mathcal{R}'(\mathcal{N})$ is a finite $\mathcal{R}'(I)$-module and hence it follows that $\mathcal{R}'(\mathcal{T})$ is a finite $\mathcal{R}'(I)$-module. Let $I$ be an $\m$-primary ideal of a $d$-dimensional analytically unramified local ring $R$ of prime characteristic $p,$ then the   \textbf{tight Hilbert function} defined by $H_I^*(n)=\ell(R/(I^n)^*)$ coincides with a  polynomial, $P_I^*(n)$, of degree $d$ for large $n$ with coefficients in $\mathbb{Q}$ and is known as the \textbf{tight Hilbert polynomial of $I$}. We write
$$P_I^*(n)=e_0^*(I)\binom{n+d-1}{d}-e_1^*(I)\binom{n+d-2}{d-1}+\cdots+(-1)^de_d^*(I),$$ where $e_i^*(I)\in \mathbb{Z}.$
Here $e^*_0(I)=e(I),$ the multiplicity of $I$,  $e^*_1(I)$ is known as the \textbf{tight Chern number of $I$} and the coefficients $e^*_i(I)$ for $i=0,1,\ldots,d$ are called the \textbf{tight Hilbert coefficients of $I.$}
The Strong Vanishing theorem for hypersurfaces plays an important role in computation of the tight closure for ``large primes."
\begin{Theorem}\label{SV_Hyp}\cite[Theorem 6.4]{Huneke1998}
Let $K$ be a field of characteristic $p>0.$ Let $f\in K[X_0,X_1,\ldots ,X_d]$ is quasihomogeneous and $R=K[X_0,X_1,\ldots,X_d]/(f).$ Assume that $R$ is an isolated singularity and the partial derivatives of $f$ with respect to $X_i$ forms a system of parameters in $R$ where $1\leq i \leq d.$ Further assume that $p>(d-1)(\deg (f))-\sum_{i=1}^d\deg(X_i).$
Let $y_1,y_2,\ldots,y_d$ be a homogeneous system of
parameters of degrees $a_1,a_2,\ldots,a_d.$ Set $A=a_1+\cdots+a_d$. Then
$$(y_1,y_2,\ldots,y_d)^*=(y_1,y_2,\ldots,y_d)+R_{\geq A}.$$
\end{Theorem}
\begin{Definition}\cite{kod_van} An $\mathbb{N}$-graded Noetherian Cohen-Macaulay ring $(S, \m),$ over a field $S_0$ of characteristic $p > 0,$ where $\m$ is the maximal graded ideal and $\dim S = d$ is said to satisfy the \textbf{Strong Vanishing Conjecture} if $$(x_1,\ldots ,x_d)^*=(x_1,\ldots ,x_d)+S_{\geq \delta}$$ for any homogeneous system of parameters $x_1,\ldots , x_d$ of $S$ where $\delta$ is the sum of degrees of $x_1,\ldots ,x_d.$
\end{Definition}
One can use the next theorem to determine the tight closure of $\m$-primary ideals.
\begin{Theorem}\label{gentytclosure} \cite[Theorem 5.11]{kod_van}
Let $R=\bigoplus_{n \geq 0} R_n$ be a $d$-dimensional finitely generated $\mathbb{N}$-graded Cohen-Macaulay ring over a field of prime characteristic $p$. Let $\m$ be its maximal homogeneous ideal and $I$ be an $\m$-primary ideal such that $R/I$ has finite projective dimension. 
Assume that $R$ has an isolated singularity at $\m$ and let 
\begin{equation}\label{eq2}
0 \to \bigoplus_{i} R(-b_{di}) \to \cdots \to \bigoplus_{i} R(-b_{1i}) \to R \to R/I \to 0.
\end{equation}
be a minimal graded free resolution of $R/I.$
Further assume that $R$ satisfies the Strong Vanishing
Conjecture. Then
\[I+R_{\geq M} \subseteq I^* \subseteq I+R_{\geq N}\]
where $N = \min_i \{b_{di}\}$ and $M = \max_i \{b_{di}\}$. 
\end{Theorem}
\begin{Remark}
Let $R=\mathbb{F}_p[X_1,X_2,\ldots ,X_{d+1}]/(X_1^r+X_2^r+\cdots +X_{d+1}^r),$ $I=(x_1,\ldots ,x_d)$ and $\m=(I,x_{d+1}).$ Then $\ell \left( \frac{R}{(I^n)^*} \right)=\ell \left( \frac{R_{\m}}{(I^nR_{\m})^*} \right),$ where the equality follows from the fact that localisation commutes with the tight closure operation for $I^n$ \cite[Corollary 3.2] {Ab_94}. Therefore the tight Hilbert polynomial of $I$ is same as that of $I_{\m}.$
\end{Remark}
\begin{Theorem}\cite[Theorem A2.60]{EisenbudGeomSyzygies}
	Let $R$ be a ring and $\alpha:F=R^f\rightarrow G=R^g$ where $f\geq g$. The Eagon-Northcott complex of a map $\alpha$ is given by
	\begin{align*}
	0\rightarrow (\sym_{f-g}G)^*\otimes \wedge^f F\xrightarrow{d_{f-g+1}} \cdots\rightarrow (\sym_2 G)^*\otimes \wedge^{g+2}F\xrightarrow{d_3}G^*\otimes \wedge^{g+1}F\xrightarrow{d_2}\wedge^gF\xrightarrow{\wedge^g\alpha}\wedge^g G.
	\end{align*}
The above complex is exact if and only if $\text{grade }I_g(\alpha)=f-g+1$.
\end{Theorem}
We refer the reader to \cite[Appendix A2H]{EisenbudGeomSyzygies} for information on the above complex and the description of the maps appearing in the complex. The main feature that we use in this article is the linearity of the maps $d_i,i\geq 2$. This can be realized from the description of the maps (see for example \cite[Example A2.69, Example A2.68]{EisenbudGeomSyzygies} and \cite[Appendix A2.6.1]{EisenbudCommAlg}). 

In a local ring $(R,\m)$ and an ideal $I$ generated by a regular sequence, the power $I^t$ can be resolved by the Eagon-Northcott complex (\cite[Page 15, after Remark 2.13]{BrunsVetter88} and the above theorem).

\section{Hilbert polynomial for the filtration $\{I_n\}_{n\geq 0},$\\ where $I_0=R$ and $I_n=I^n+\m^{n+t}$ for $n\geq 1$} \label{sec_HPfilt}
In the subsequent sections, we will show that in most of the cases we consider, the tight closure of powers of a linear system of parameters  in diagonal hypersurfaces takes a special form. Therefore computing the Hilbert polynomial in such cases becomes crucial.\\
Let $R=K[X_1,X_2,\ldots,X_{d+1}]/(X_1^r+X_2^r+\cdots+X_{d+1}^r),$ $I=(x_1,x_2,\ldots, x_d)$ and $\m=(I,x_{d+1}),$ where $K$ is a field. Let $\mathcal{F}=\{I_n\}_{n\geq 0},$ where $I_n=\{I^n+\m^{n+t}\}_{n\geq 0}$ for fixed $t\in \mathbb{N}.$ Then we calculate the Hilbert polynomial $P_{\mathcal{F}}(n).$
\begin{Proposition}\label{HP_t}
Let $T=K[X_1,X_2,\ldots,X_{d+1}]$ and $R=T/(f)$ where $K$ is an infinite field and $f\in T$ is a homogeneous square-free polynomial of degree $r\geq 2.$ Let $I$ be a minimal reduction of $\m=(x_1,\ldots ,x_{d+1})R.$ Let $\mathcal{F}=\{I_n\}_{n\geq 0},$ where $I_0=R,$ $I_n=I^n+\m^{n+t}$ for all $n\geq 1$ and for some fixed $t\in \mathbb{N}.$ Then\\
{\rm (1)} $\gr_{\mathcal{F}}(R)$ is Cohen-Macaulay.\\
{\rm (2)} If $r\leq t+1$ then $I_n=I^n$ for all $n$ and $P_{\mathcal{F}}(n)=r\binom{n+d-1}{d}.$\\
{\rm (3)} If $r\geq t+2$ then
$$H(\gr_{\mathcal{F}}(R),\lambda)=\frac{(t+1)+\l+\cdots+\l^{r-t-1}}{(1-\l)^{d}},$$ $$e_j(\mathcal{F})=\binom{r-t}{j+1} \text{ for } 1\leq j \leq r-t-1 \text{ and } e_j(\mathcal{F})=0 \text{ for }r-t \leq j \leq d.$$ Also $\eta(\mathcal{F})=r-t-d-1$ and $r(\mathcal{F})=r-t-1.$\\
\end{Proposition}
\begin{proof}
We note that $\mathcal{F}$ is an $I$-admissible filtration. Since $f$ is square-free in $T,$ $R_{\m}$ is reduced and analytically unramified. Clearly $I_n\subseteq \overline{I^n}=\m ^n$ which implies that
$$\mathcal{R'}(I)\hookrightarrow \mathcal{R'}(\mathcal{F})\hookrightarrow \overline{\mathcal{R'}}(I).$$ As $R_{\m}$ is analytically unramified, $\overline{\mathcal{R'}}(I)$ is a finite $\mathcal{R'}(I)$-module so that $\mathcal{R'}(\mathcal{F})$ is a finite $\mathcal{R'}(I)$-module making $\mathcal{F}$ an $I$-admissible filtration.\\
{\rm (1)} Since $R$ is Cohen-Macaulay, by Valabrega-Valla Theorem \cite[Proposition 3.5]{HM}, it is enough to show that $I \cap I_n =II_{n-1}$ for $n\geq 1.$
Notice that $I\cap \m ^n=I \m^{n-1}$ for $n\geq 1$ as $\gr_{\m}(R)\cong R$ is Cohen-Macaulay. This implies that for all $n\geq 1,$
$$I \cap I_n =I^n+I\cap \m^{n+t}=I^n+I\m^{n+t-1}=II_{n-1}.$$ Therefore $\gr_{\mathcal{F}}(R)$ is Cohen-Macaulay.\\
{\rm (2)} Let $r\leq t+1.$ We will show that $r(\m)=r-1.$ The Hilbert series of $R$ is given by $$H(R,\l)=\frac{1-\l ^r}{(1-\l)^{d+1}}=\frac{1+\l +\cdots +\l ^{r-1}}{(1-\l)^d}.$$ Since $\gr_{\m}(R)$ is Cohen-Macaulay and $\eta(\m)=r-1-d,$ we have $r(\m)=r-1$ \cite[Corollary 5.7]{marthesis}.\\
Let $r\leq t+1.$ Since $r(\m)=r-1,$ for $n\geq 1,$ $$I_n=I^n+\m^{n+t}=I^n+I^n\m ^t=I^n.$$ This implies that $\mathcal{F}$ is the $I$-adic filtration and for $n\geq 1,$ $$H_{\mathcal{F}}(n)=P_{\mathcal{F}}(n)=r\binom{n+d-1}{d}.$$\\
{\rm (3)} Let $r\geq t+2.$ Apply induction on $d.$ 
If $d=0$ then $R=\frac{K[X_1]}{(X_1^r)},$ $I=(0),$ $\m=(x_1)$ and $\mathcal{F}=\{I_n\}_{n\geq 0}$ where $I_0=R$ and $I_n=\m ^{n+t}$ for $n\geq 1.$ In this case $\gr_{\mathcal{F}}(R)=\frac{R}{\m ^{t+1}}\bigoplus_{n\geq 1}\frac{\m^{n+t}}{\m^{n+t+1}}.$ This implies that
\begin{align*}
H(\gr_{\mathcal{F}}(R),\lambda)&=\sum_{n\geq 0}\ell ([\gr_{\mathcal{F}}(R)]_n)\l^n\\
&=\ell\left(\frac{R}{\m^{t+1}} \right)+\ell\left(\frac{\m^{t+1}}{\m^{t+2}} \right)\l +\cdots +\ell\left(\frac{\m^{r-1}}{\m^{r}} \right)\l^{r-1-t}.
\end{align*}
Note that $\frac{R}{\m^{t+1}}\cong \frac{K[X_1]}{(X_1^{t+1})}$ since $r\geq t+2\geq t+1.$ Therefore $\ell\left(\frac{R}{\m^{t+1}} \right)=t+1.$ If $1\leq i \leq r-1-t$ then $$\frac{R}{\m^{t+i}}\cong \frac{K[X_1]}{(X_1^{t+i})}\text{ and } \frac{R}{\m^{t+i+1}}\cong \frac{K[X_1]}{(X_1^{t+i+1})}.$$
This implies that $\ell \left(\frac{\m^{t+i}}{\m^{t+i+1}}\right)=1$ for $1\leq i \leq r-t-1.$ Therefore $$H(\gr_{\mathcal{F}}(R),\lambda)=(t+1)+\l+\cdots+\l^{r-t-1}.$$ Let $d\geq 1.$ After a linear change of coordinates, we may assume that $f$ is monic in $X_{d+1}.$ Since $f$ is monic in $X_{d+1},$ $(x_1,\ldots, x_d)$ is a minimal reduction of $\m.$ Thus we may assume that $I=(x_1,\ldots ,x_d).$ Clearly $x_1^*$ is a nonzerodivisor on $\gr_{\mathcal{F}}(R)$ so that $$\frac{\gr_{\mathcal{F}}(R)}{x_1^* \gr_{\mathcal{F}}(R)}\cong \gr_{\mathcal{F}/(x_1)}(R/(x_1)).$$ Let $S=R/(x_1),$ $J=IS$ and $\Gamma=\{J_n\}_{n\geq 0}$ where $J_0=S$ and $J_n=J^n+\mathfrak{n}^{n+t}$ for $n\geq 1$ where $\n$ is the unique maximal homogeneous ideal in the $(d-1)$-dimensional ring $S.$ Notice that the filtration $\Gamma$ is same as the filtration $\mathcal{F}/(x_1)=\{L_n\}_{n\geq 0}$ in $S.$ Indeed for $n\geq 1,$ $J_n=J^n+\n ^{n+t}=\frac{I^n+\m ^{n+t}+(x_1)}{(x_1)}$ and $L_n=\frac{I_n+(x_1)}{(x_1)}=\frac{I^n+\m ^{n+t}+(x_1)}{(x_1)}.$ Also $J_0=S$ and $L_0=\frac{I_0+(x_1)}{(x_1)}=R/(x_1)=S.$
By induction hypothesis, $$H(\gr_{\Gamma}(S),\lambda)=\frac{(t+1)+\l+\cdots+\l^{r-t-1}}{(1-\l)^{d-1}}.$$ Since $x_1^*$ is a nonzerodivisor of $\gr_{\mathcal{F}}(R),$ it follows that $$H(\gr_{\mathcal{F}}(R),\lambda)=\frac{(t+1)+\l+\cdots+\l^{r-t-1}}{(1-\l)^{d}}.$$ Let $h(\l)=(t+1)+\l+\cdots+\l^{r-t-1}.$ Then by \cite[Proposition 4.1.9]{BH} $e_j(\mathcal{F})=\frac{h^{(j)}(1)}{j!}.$ Note that if $r-t\leq j\leq d$ then $e_j(\mathcal{F})=0.$ 
Let $1\leq j \leq r-t-1$ then $h^{(j)}(\l)=\sum_{i\geq j}^{r-t-1}i(i-1)\cdots (i-(j-1))\l^{i-j}.$ Therefore for $1\leq j \leq r-t-1,$ $$e_j(\mathcal{F})=\sum_{i\geq j}^{r-t-1}\frac{i(i-1)\cdots (i-(j-1))}{j!}=\sum_{i\geq j}^{r-t-1}\binom{i}{j}=\binom{r-t}{j+1}.$$ Since $\gr_{\mathcal{F}}(R)$ is Cohen-Macaulay, $\eta(\mathcal{F})=r-t-d-1$ and $r(\mathcal{F})=r-t-1.$
\end{proof}
\section{The Tight Hilbert polynomial for diagonal hypersurfaces}\label{sec_THP_diag_Hy}
The purpose of this section is to find the tight Hilbert polynomial in certain diagonal hypersurface rings. As evidenced in \cite{GMV2019}, this will enable us to detect $F$-rationality of such rings. The first proposition helps us to compute the tight Hilbert polynomial of a system of parameters generated by test elements in excellent reduced Cohen-Macaulay rings. In fact, it also provides information about the various blow up algebras associated with the tight closure filtration in this case. It is an easy consequence of \cite[Lemma 3.1]{AHS}.
\begin{Proposition}\label{ahsappl}
Let $(R,\m)$ be $d$-dimensional excellent reduced Cohen-Macaulay local ring. Let $x_1,x_2,\ldots,x_d$ be test elements and  $Q=(x_1,x_2,\ldots,x_d)$ be $\m$-primary.Then\\
{\rm (1) }$(Q^n)^*=Q(Q^{n-1})^*$ for all $n\geq 2.$\\
{\rm (2) }$e_0(Q)-e_1^*(Q)=\ell(R/Q^*).$\\
{\rm (3) }$\ell(R/(Q^n)^*)=e_0(Q)\binom{n+d-1}{d}-e_1^*(Q)\binom{n+d-2}{d-1}$ for all $n\geq 1.$ \\
{\rm (4) } $G^*(Q)=\bigoplus_{n\geq 0}(Q^n)^*/(Q^{n+1})^*$ and $\mathcal{R}^*(Q)=\bigoplus_{n\geq 0}(Q^n)^*t^n$ are Cohen-Macaulay.
\end{Proposition}
\begin{proof}
By \cite[Lemma 3.1]{AHS}, $(Q^n)^*=Q^{n-1}Q^*.$ This implies that $$(Q^n)^*\subseteq Q((Q^{n-2})^*Q^*)^* =Q(Q^{n-1})^*\subseteq (Q^n)^*\text{ for all } n\geq 2.$$ This proves {\rm (1)}. The remaining parts follow by Huneke-Ooishi Theorem \cite{H1987}.
\end{proof}
In the next theorem, we demonstrate the use of Strong Vanishing Theorem for computing the tight Hilbert polynomial. It is well known that the Strong Vanishing Theorem helps in computation of the tight closure of an ideal generated by system of parameters for rings with ``large" prime characteristic. However to find the tight Hilbert polynomial, we need to compute the tight closure of powers of such ideals. For this, we use Theorem \ref{gentytclosure} which ensures a bound for tight closure of $\m$-primary ideals once we have its projective resolution. This resolution in our case will be given by the Eagon-Northcott complex.

\begin{Theorem}\label{THP_SVT}
Let $T=\mathbb{F}_p[X_1,X_2,\ldots,X_{d+1}],$ $R=T/(f)$ where $f\in T$ is a square-free, monic in $X_{d+1}$ homogeneous ploynomial of degree $d+1$ and $p>(d-1)(d+1)-d.$ Let $I=(x_1,\ldots ,x_d)$ and $\m=(I,x_{d+1}).$ Then $(I^n)^*=I^n+\m^{n+d-1}$ and for $n\geq 1,$
$$\ell \left( \frac{R}{(I^n)^*}\right)=(d+1)\binom{n+d-1}{d}-\binom{n+d-2}{d-1}.$$ Moreover $\gr ^*_I(R)$ is Cohen-Macaulay.
\end{Theorem}
\begin{proof}
Since $p>(d-1)(d+1)-d,$ it follows from Theorem \ref{SV_Hyp} that Strong Vanishing holds for $R.$ Therefore we have $I^*=I+\m ^d .$ Using the fact that the Eagon-Northcott complex resolves $I^n$ and that the differentials $d_i,i\geq 2$ of this complex is linear, it follows from Theorem \ref{gentytclosure} that $(I^n)^*=I^n+\m ^{n+d-1}.$ We claim that $r^*(I)\leq 1.$ Indeed, for $n\geq 1,$ $I(I^n)^*=I^{n+1}+I \m ^{n+d-1}=I^{n+1}+\m ^{n+d}$ where the last equality follows as $r(\m)=d.$
Therefore by Huneke-Ooishi Theorem (Proposition \ref{ahsappl}) we have for $n\geq 1,$
\begin{align*}
\ell \left( \frac{R}{(I^n)^*}\right)&=e_0(I)\binom{n+d-1}{d}-e_1^*(I)\binom{n+d-2}{d-1}\\
&=(d+1)\binom{n+d-1}{d}-e_1^*(I)\binom{n+d-2}{d-1}
\end{align*}
Put $n=1$ to obtain $\ell\left( \frac{R}{I+\m ^d} \right)=(d+1)-e_1^*(I).$ Since
$$\frac{R}{I+\m ^d}\cong \frac{\mathbb{F}_p[X_1,X_2,\ldots,X_{d+1}]/(f)}{\left((X_1, X_2,\ldots ,X_d)+(X_1,X_2, \ldots , X_{d+1}) ^d +(f)\right)/(f)}\cong \frac{\mathbb{F}_p[X_{d+1}]}{({X_{d+1}})^d},$$ it follows that $e_1^*(I)=1.$
\end{proof}
Due to \cite[Proposition 5.21 (c)]{HR}, many diagonal hypersurface rings are $F$-pure, namely, $R=\mathbb{F}_p[X_1,X_2,\ldots,X_{d+1}]/(X_1^{d+1}+X_2^{d+1}+\cdots +X_{d+1}^{d+1}),$ where $ p\equiv 1 ( \mod d+1).$ In such cases, the test ideal turns out to be a radical ideal by Proposition \ref{Prop_FedWat_test}.Using this, we find the tight Hilbert polynomial in $F$-pure diagonal hypersurface rings where the number of variables is equal to the degree of the defining equation.

\begin{Theorem}\label{THP_Fpure}
Let $R=\mathbb{F}_p[X_1,X_2,\ldots,X_{d+1}]/(X_1^{d+1}+X_2^{d+1}+\cdots+X_{d+1}^{d+1})$ be F-pure, $I=(x_1,\ldots ,x_d)$ and $\m=(I,x_{d+1})$. Then for all $n\geq 1,$
$$\ell \left( \frac{R}{(I^n)^*}\right)=(d+1)\binom{n+d-1}{d}-\binom{n+d-2}{d-1}.$$ Moreover $\gr ^*_I(R)$ is Cohen-Macaulay.
\end{Theorem}
\begin{proof}
{Observe that, using Theorem \ref{Thm:Hoc_testele}, $x_1^d,\ldots ,x_{d+1}^d\in \tau(R).$ Since $R$ is F-pure, taking radicals on both sides, Proposition \ref{Prop_FedWat_test} yields $\m \subseteq \tau(R).$} Therefore by Proposition \ref{ahsappl} it follows that for $n\geq 1,$
$$\ell \left( \frac{R}{(I^n)^*}\right)=(d+1)\binom{n+d-1}{d}-e_1^*(I)\binom{n+d-2}{d-1},$$ where $e_1^*(I)=\ell\left( \frac{I^*}{I} \right).$ Using Brian\c{c}on-Skoda Theorem $x_{d+1}^d\in I^*\setminus I$ which implies that $\tau(R)=\m.$ We claim that $I^*=I+\m ^d.$ Since $I:\tau(R)=I^*$ \cite[Corollary 4.2]{Huneke1998}, it is enough to show that $x_{d+1}^{d-1}\notin I^*.$ If possible, let $x_{d+1}^{d-1}\in I^*$ then $x_{d+1}^d\in I,$ a contradiction. Hence the claim. Since $$\frac{R}{I+\m^d}\cong \frac{\mathbb{F}_p[X_1,\ldots ,X_{d+1}]}{(X_1,\ldots ,X_d,X_{d+1}^d)}\cong \frac{\mathbb{F}_p[X_{d+1}]}{(X_{d+1}^d)},$$ it follows that $e_1^*(I)=1.$
\end{proof}
As promised in section \ref{sec_HPfilt}, we now give a class of rings where the tight closure filtration turns out to be the filtration $\{I_n\}_{n\geq 0},$ where $I_0=R$ and $I_n=I^n+\m^{n+t}$ for $n\geq 1.$
\begin{Theorem}
Let $R=\mathbb{F}_p[X_1,X_2,\ldots,X_{t+2}]/(X_1^{r}+X_2^{r}+\cdots +X_{t+2}^{r}),$ with $p\nmid r,$  $I=(x_1,\ldots ,x_{t+1})$ and $\m=(I,x_{t+2}),$ where $p>tr-(t+1).$ Then\\
{\rm (1)} $\gr_I^*(R)$ is Cohen-Macaulay.\\
{\rm (2)} The tight Hilbert polynomial $P_I^*(n)$ is given by  \[
    P_I^*(n)= 
\begin{cases}
    r\binom{n+t}{t+1},& \text{ if } r\leq t+1,\\
    r\binom{n+t}{t+1}-\binom{r-t}{2}\binom{n+t-1}{t}+\binom{r-t}{3} \binom{n+t-2}{t-1}+\cdots +(-1)^j\binom{r-t}{j+1}\binom{n+t-j}{t+1-j},& \text{ if } r\geq t+2.\\
\end{cases}
\]
If $r\geq t+2,$ $e_j^*(I)=0$ for $r-t\leq j \leq t+1.$
\end{Theorem}
\begin{proof}
Since $p>tr-(t+1),$ it follows from Theorem \ref{SV_Hyp} that Strong Vanishing Conjecture holds for $R.$ Therefore we have $I^*=I+\m ^{t+1} .$ Using the Eagon-Northcott complex, it follows from Theorem \ref{gentytclosure} that $(I^n)^*=I^n+\m ^{n+t}.$ The result follows from Proposition \ref{HP_t}.
\end{proof}  
\section{The Tight Hilbert polynomial for degree $p+1$ hypersurface ring of characteristic $p$} \label{sec_THP p+1 degree}
The goal of this section is to compute the tight Hilbert polynomial of linear system of parameters in  diagonal hypersurface rings where the characteristic of the ring is one less than the degree of the defining equation. In order to find the tight Hilbert polynomial, we calculate the initial ideals of certain ideals in $K[X_1,\ldots ,X_n].$\begin{Lemma}\label{comp_initialideal}
Let $A=K[X_1,X_2,\ldots ,X_m]$ be a polynomial ring over field $K$. Let $k\in \mathbb{N}$ and $q=p^e$ for some $e.$ Let $J=(X_1^N+X_2^N+\cdots +X_m^N,(X_2^q,X_3^q,\ldots ,X_m^q)^k).$ If $>$ denotes the graded reverse lex ordering with $X_1>X_2>\cdots >X_m,$ then $$\ini _> (J)=(X_1^N,(X_2^q,X_3^q,\ldots ,X_m^q)^k).$$
\end{Lemma}
\begin{proof}
It suffices to show that $$G=\Bigg \{ \sum_{i=1}^mX_i^N, \left[\prod_{j=2}^{m-1} X_j^{qr_j}\right]X_m^{q[k-(r_2+r_3+\cdots +r_{m-1})] } \text{, }r_j\in \{0,1,2,\ldots ,k\} \text{ for } 2\leq j\leq m-1  \Bigg \}$$ is a Gr\"{o}bner basis of $J$.
For $r_j\in \{0,1,\ldots ,k\},$ consider the $S$-polynomial of $f=X_1^N+X_2^N+\cdots +X_m^N$ and $g=X_2^{qr_2}X_3^{qr_3}\cdots X_{m-1}^{qr_{m-1}} X_m^{q[k-(r_2+r_3+\cdots +r_{m-1})]},$
\begin{align*}
S(f,g)&=\frac{X_1^N\left[\prod_{j=2}^{m-1} X_j^{qr_j}\right] X_m^{q[k-(r_2+r_3+\cdots +r_{m-1})]}  }{X_1^N}f -\frac{X_1^N\left[\prod_{j=2}^{m-1} X_j^{qr_j}\right] X_m^{q[k-(r_2+r_3+\cdots +r_{m-1})]}}{\left[\prod_{j=2}^{m-1} X_j^{qr_j}\right] X_m^{q[k-(r_2+r_3+\cdots +r_{m-1})]}} g\\
&=\left[\prod_{j=2}^{m-1} X_j^{qr_j}\right] X_m^{q[k-(r_2+r_3+\cdots +r_{m-1})]}\left(\sum_{i=2}^mX_i^N\right).
\end{align*}
By division algorithm, it follows that $\overline{S(f,g)}^G=0$ for $f=X_1^N+X_2^N+\cdots +X_m^N$ and\\ $g=X_2^{qr_2}X_3^{qr_3}\cdots X_{m-1}^{qr_{m-1}} X_m^{q[k-(r_2+r_3+\cdots +r_{m-1})]},$ where $r_j\in \{0,1,\ldots ,k\}.$ Now for $\overline{r}=(r_2,r_3,\ldots, r_{m-1})$ and $\overline{s}=(s_2,s_3,\ldots, s_{m-1})$ with $\overline{r}\neq \overline{s},$ we have $S$-polynomial,
\begin{align*}
&S\left(\left[\prod_{j=2}^{m-1} X_j^{qr_j}\right] X_m^{q[k-(r_2+r_3+\cdots +r_{m-1})]},\left[\prod_{j=2}^{m-1} X_j^{qs_j}\right] X_m^{q[k-(s_2+s_3+\cdots +s_{m-1})]}\right)=0.
\end{align*}
It follows that $G$ is a Gr\"{o}bner basis of $J$ and hence $\ini _> (J)=(X_1^N,(X_2^q,X_3^q,\ldots ,X_m^q)^k).$
\end{proof}
\begin{Theorem}\label{tytclIindegp+1}
Let $p$ be an odd prime, $R=\frac{\mathbb{F}_p[X_1,X_2,\ldots , X_{d+1}]}{(X_1^{p+1}+X_2^{p+1}+\cdots +X_{d+1}^{p+1})},$  $I=(x_1,x_2,\ldots ,x_d)$ and $\m=(I,x_{d+1}).$ Then $$I^*=I+\m ^2.$$
\end{Theorem}
\begin{proof}
First we will show that $x_{d+1}^2\in I^*.$ Clearly, 
\begin{equation} \label{eqn chp deg p+1}
(x_{d+1}^{p+1})^{p^e}=-(x_{1}^{p+1})^{p^e}-(x_{2}^{p+1})^{p^e}-\cdots -(x_{d}^{p+1})^{p^e}\in I^{[p^{e+1}]}.
\end{equation}
Therefore $\underbrace{(x_{d+1}^2)^{p^e}\cdots (x_{d+1}^2)^{p^e}\,}_\text{$m$ times} \in I^{[p^{e+1}]}$ where $p+1=2m.$ Since $p>m,$ it follows that $$(x_{d+1}^2)^{p^{e+1}}=\underbrace{(x_{d+1}^2)^{p^e}\cdots (x_{d+1}^2)^{p^e}\,}_\text{$p$ times}\in I^{[p^{e+1}]}.$$ Hence $x_{d+1}^2\in I^*.$ 
It is enough to show that $x_{d+1} \notin I^* .$ {As $x_d^{p}$ is a test element by Theorem \ref{Thm:Hoc_testele}, we show that $x_{d}^{p}x_{d+1}^{p^3}\notin I^{[p^3]}.$} If possible, let $x_d^px_{d+1}^{p^3}\in I^{[p^3]}.$ Note that $x_d^px_{d+1}^{p^3}\in I^{[p^3]}$ if and only if $X_d^{p}X_{d+1}^{p^3}\in J,$ where $J=(X_1^{p+1}+\cdots+X_{d+1}^{p+1},X_1^{p^3},\ldots ,X_d^{p^3})$ in $\mathbb{F}_p[X_1,\ldots ,X_{d+1}].$ 
Let $>$ denote the graded reverse lex ordering with $X_{d+1}>X_1>X_2>\cdots >X_d$ in $\mathbb{F}_p[X_1,X_2,\ldots ,X_{d+1}].$ Let $p^3=(p+1)u+i,$ where $1\leq i \leq p.$ Note that $i\neq 0.$ Clearly,
$$X_d^p X_{d+1}^{p^3}=X_d^pX_{d+1}^{(p+1)u+i}\equiv f (\text{mod }J),$$ where $f= (-1)^uX_d^p(X_1^{p+1}+\cdots +X_d^{p+1})^uX_{d+1}^i.$ Note that $\ini_{>}(f)=X_d^p X_1^{(p+1)u}X_{d+1}^i$ and by Lemma \ref{comp_initialideal} it follows that $\ini_{>}(J)=(X_{d+1}^{p+1},X_1^{p^3},\ldots, X_d^{p^3}).$ The exponent of $X_1$ in $\ini_{>}(f)$ is $(p+1)u=p^3-i<p^3,$ exponent of $X_d$ and $X_{d+1}$ in $\ini_{>}(f)$ is $p$ and $i$ respectively. This implies that $\ini_{>}(f)\notin \ini_{>}(J).$ Hence $x_{d+1}\notin I^*$ so that $I^*=I+\m ^2.$
\end{proof}
\begin{Lemma}\label{lemma_belonging to power of I}
Let $R=\frac{\mathbb{F}_p[X_1,X_2,\ldots , X_{d+1}]}{(X_1^{p+1}+X_2^{p+1}+\cdots +X_{d+1}^{p+1})}$ and $I=(x_1,x_2,\ldots ,x_d)$ where $p$ is prime. 
\[
    (x_{d+1}^n)^{[p^{e+1}]}\in  
\begin{cases}
    (I^{n-1})^{[p^{e+1}]},& \text{ if } n\not\equiv 0 \pmod {p+1},\\
    (I^{n})^{[p^{e+1}]}, & \text{ if } n\equiv 0 \pmod {p+1}.
\end{cases}
\]
\end{Lemma}
\begin{proof}
Let $n\equiv 0 \pmod {p+1}.$ Then $n=(p+1)m$ for some $m\in \mathbb{N}.$ Then
$$(x_{d+1}^n)^{p^{e+1}}=(x_{d+1}^{(p+1)m})^{p^{e+1}}=[(x_{d+1}^{p+1})^{p^{e+1}}]^m \in \left((I^{p+1})^{[p^{e+1}]}\right)^m=(I^n)^{[p^{e+1}]}.$$
Next let $n\not\equiv 0 \pmod {p+1}.$ Write $n=(p+1)m+r$ where $1\leq r \leq p.$ If $m=0$ then we may assume that $r>1.$ 
\begin{align*}
(x_{d+1}^r)^{p^{e+1}}&=\underbrace{(x_{d+1}^r)^{p^e}\cdots (x_{d+1}^r)^{p^e}\,}_\text{$p$ times}=\underbrace{(x_{d+1}^p)^{p^e}\cdots (x_{d+1}^p)^{p^e}\,}_\text{$r$ times}\\
&=\underbrace{(x_{d+1}^{p+1})^{p^e}\cdots (x_{d+1}^{p+1})^{p^e}\,}_\text{$r-1$ times}(x_{d+1}^{p-(r-1)})^{p^e} \in (I^{r-1})^{[p^{e+1}]},
\end{align*}
where the last statement follows from equation \ref{eqn chp deg p+1} in Theorem \ref{tytclIindegp+1}. This implies that $$(x_{d+1}^n)^{[p^{e+1}]}=(x_{d+1}^{m(p+1)}x_{d+1}^r)^{p^{e+1}}\in (I^{(m(p+1))}I^{r-1})^{[p^{e+1}]}= (I^{n-1})^{[p^{e+1}]}.$$
\end{proof}

\begin{Theorem}\label{p+1 deg ch p}
Let $R=\frac{\mathbb{F}_p[X_1,X_2,\ldots , X_{d+1}]}{(X_1^{p+1}+X_2^{p+1}+\cdots +X_{d+1}^{p+1})},$ $I=(x_1,x_2,\ldots ,x_d)$ and $\m=(I,x_{d+1}).$ Then\\
\rm {(1)} $(I^n)^*\supseteq I^n+\m ^{n+1};$\\
\rm {(2)} If $x_1^{a_1}x_2^{a_2}\cdots x_{d+1}^{a_{d+1}}\in \m^n\setminus I^n$ then $x_1^{a_1}x_2^{a_2}\cdots x_{d+1}^{a_{d+1}}\notin (I^n)^*.$ \\
\rm {(3)} $x_1^{a_1}x_2^{a_2}\cdots x_{d+1}^{a_{d+1}}\notin (I^n)^*$ for $\sum_{j=1}^{d+1}a_j<n.$
\end{Theorem}
\begin{proof}
\rm {(1)} Consider $x_1^{b_1}x_2^{b_2}\cdots x_{d+1}^{b_{d+1}}\in \m ^{n+1}$ where $\sum_{j=1}^{d+1}b_j=n+1$. Using Lemma \ref{lemma_belonging to power of I}, observe that 
$$
(x_1^{b_1}x_2^{b_2}\cdots x_{d+1}^{b_{d+1}})^{p^{e+1}}=(x_1^{b_1}x_2^{b_2}\cdots x_d^{b_d})^{p^{e+1}} (x_{d+1}^{b_{d+1}})^{p^{e+1}} 
$$ \[
\in
\begin{cases}
    (I^{n+1-b_{d+1}})^{[p^{e+1}]} (I^{b_{d+1}-1})^{[p^{e+1}]},& \text{if } b_{d+1} \not\equiv 0 \pmod {p+1},\\
    (I^{n+1-b_{d+1}})^{[p^{e+1}]} (I^{b_{d+1}})^{[p^{e+1}]},& \text{if } b_{d+1} \equiv 0 \pmod {p+1}.
\end{cases}
\]
In either case, $(x_1^{b_1}x_2^{b_2}\cdots x_{d+1}^{b_{d+1}})^{p^{e+1}}\in (I^n)^{[p^{e+1}]}.$ This implies that $I^n+\m ^{n+1} \subseteq (I^n)^*.$\\
\rm {(2)} Note that $x_1^{a_1}x_2^{a_2}\cdots x_{d+1}^{a_{d+1}}\in \m^n\setminus I^n$ implies that $a_{d+1}\geq 1$ and that $p+1\nmid a_{d+1}.$ Indeed, if $a_{d+1}=(p+1)u$ for some $u$ then $$x_1^{a_1}x_2^{a_2}\cdots x_{d+1}^{a_{d+1}}=(-1)^ux_1^{a_1}x_2^{a_2}\cdots x_d^{a_d}(x_1^{p+1}+\cdots +x_{d}^{p+1})^u\in I^n,$$ a contradiction to the assumption. Hence $p+1\nmid a_{d+1}.$ {We use the fact that $x_d^p$ is a test element using Theorem \ref{Thm:Hoc_testele}.} Observe that $x_1^{a_1}x_2^{a_2}\cdots x_{d+1}^{a_{d+1}}\notin (I^n)^*$ if and only if $x_d^p(x_1^{a_1}x_2^{a_2}\cdots x_{d+1}^{a_{d+1}})^q\notin (I^n)^{[q]}$ for some $q.$
For $e\in \mathbb{N}$ set $q=p^{e+1}>p+1$ and $a_{d+1}q=(p+1)u+i$ where $1\leq i \leq  p.$ Let $>$ be the graded reverse lex ordering with $X_{d+1}>X_1>X_2>\cdots >X_d$ in $A=\mathbb{F}_p[X_1,X_2,\ldots ,X_d,X_{d+1}].$\\ Let $J=(X_1^{p+1}+X_2^{p+1}+\cdots +X_{d+1}^{p+1},(X_1^q,X_2^q,\ldots, X_d^q )^n)$ be an ideal in $A.$ Using Lemma \ref{comp_initialideal}, it follows that $\ini _>(J)=(X_{d+1}^{p+1},(X_1^q,X_2^q,\ldots, X_d^q )^n).$ Now
\begin{align*}
{X_d}^p({X_1}^{a_1q}\cdots {X_{d+1}}^{a_{d+1}q})&={X_1}^{a_1q}\cdots  X_{d-1}^{a_{d-1}q}X_d^{a_dq+p}{X_{d+1}^{(p+1)u+i}}\\
&\equiv (-1)^uX_1^{a_1q}\cdots  X_{d-1}^{a_{d-1}q}X_d^{a_dq+p}(X_1^{p+1}+\cdots +X_d^{p+1})^uX_{d+1}^i \;(\text{mod }J). 
\end{align*}
Let $f=(-1)^uX_1^{a_1q}\cdots  X_{d-1}^{a_{d-1}q}X_d^{a_dq+p}(X_1^{p+1}+\cdots +X_d^{p+1})^uX_{d+1}^i.$ Clearly $$X_d^p(X_1^{a_1q}X_2^{a_2q}\cdots X_{d+1}^{a_{d+1}q})\in J \Leftrightarrow f\in J.$$ Note that $\ini _>(f)=X_1^{a_1q+(p+1)u}X_2^{a_2q}\cdots  X_{d-1}^{a_{d-1}q}X_d^{a_dq+p}X_{d+1}^i.$
The exponent of $X_1$ in $\ini _>(f)$ is $$a_1q+a_{d+1}q-i=(a_1+a_{d+1}-1)q+(q-i)=[n-(\sum_{j=2}^da_j)-1]q+(q-i).$$ The exponents of $X_2,\ldots ,X_d$ in $\ini _>(f)$ are $a_2q,\ldots,a_{d-1}q,a_dq+p$ respectively. Since $p+1<q,$ the sum of multiples of $q$ in exponents of $X_1,X_2,\ldots ,X_d$ is $(n-1)q.$ We have $\ini _>(f)\notin \ini _>(J).$ This implies that ${x_d}^p({x_1}^{a_1}{x_2}^{a_2}\cdots {x_{d+1}}^{a_{d+1}})^q\notin (I^n)^{[q]}.$ Hence $x_1^{a_1}x_2^{a_2}\cdots x_{d+1}^{a_{d+1}}\notin (I^n)^*.$\\
\rm {(3)} Now we show that $x_1^{a_1}x_2^{a_2}\cdots x_{d+1}^{a_{d+1}}\notin (I^n)^*$ for $\sum_{j=1}^{d+1}a_j<n.$ Observe that $$x_1^{a_1}x_2^{a_2}\cdots x_{d+1}^{a_{d+1}}\notin (I^n)^* \Leftrightarrow x_d^p(x_1^{a_1}x_2^{a_2}\cdots x_{d+1}^{a_{d+1}})^q\notin (I^n)^{[q]}=(x_1^q,x_2^q,\ldots ,x_{d}^q)^n$$ for some $q.$ For $e\in \mathbb{N}$ set $q=p^{e+1}>p+1$ and $a_{d+1}q=(p+1)u+i$ where $0\leq i \leq  p.$ Let $>$ be the graded reverse lex ordering with $X_{d+1}>X_1>X_2>\cdots >X_d$ in $A=\mathbb{F}[X_1,X_2,\ldots ,X_d].$ Let $J=(X_1^{p+1}+X_2^{p+1}+\cdots +X_{d+1}^{p+1},(X_1^q,X_2^q,\ldots, X_d^q )^n)$ be an ideal in $A.$ Using Lemma \ref{comp_initialideal}, it follows that $\ini _>(J)=(X_{d+1}^{p+1},(X_1^q,X_2^q,\ldots, X_d^q )^n).$ Using the arguments as above, we get
\begin{align*}
X_d^p(X_1^{a_1q}\cdots X_{d+1}^{a_{d+1}q})\equiv (-1)^u{X_1}^{a_1q}\cdots  X_{d-1}^{a_{d-1}q}X_d^{a_dq+p}(X_1^{p+1}+\cdots +X_d^{p+1})^uX_{d+1}^i \;(\text{mod }J).  
\end{align*}
Let $f=(-1)^uX_1^{a_1q}\cdots  X_{d-1}^{a_{d-1}q}X_d^{a_dq+p}(X_1^{p+1}+\cdots +X_d^{p+1})^uX_{d+1}^i.$ Clearly $$X_d^p(X_1^{a_1q}X_2^{a_2q}\cdots X_{d+1}^{a_{d+1}q})\in J \Leftrightarrow f\in J.$$ Note that $\ini _>(f)=X_1^{a_1q+(p+1)u}X_2^{a_2q}\cdots  X_{d-1}^{a_{d-1}q}X_d^{a_dq+p}X_{d+1}^i.$
The exponent of $X_1$ in $\ini _>(f)$ is $$a_1q+a_{d+1}q-i=(a_1+a_{d+1}-1)q+(q-i).$$ The exponents of $X_2,\ldots ,X_d$ in $\ini _>(f)$ is $a_2q,\ldots,a_{d-1}q,a_dq+p$ respectively. The sum of multiples of $q$ in exponents of $X_1,X_2,\ldots ,X_d$ is $$(a_1+a_{d+1})q+a_2q+\cdots +a_{d}q \text{ or }(a_1+a_{d+1}-1)q+a_2q+\cdots +a_{d}q$$ depending on $i=0$ or $i\neq 0$ respectively. Since $p+1<q,$ in either case the sum of multiples of $q$ in exponents of $X_1,X_2,\ldots ,X_d$ in $\ini_{>}(f)<nq.$ Therefore $\ini _>(f)\notin \ini _>(J).$ Thus $x_d^p(x_1^{a_1}x_2^{a_2}\cdots x_{d+1}^{a_{d+1}})^q\notin (I^n)^{[q]}$ for some $q$ which implies that $x_1^{a_1}x_2^{a_2}\cdots x_{d+1}^{a_{d+1}}\notin (I^n)^*.$
\end{proof}
\begin{Theorem}\label{THP ch p deg p+1}
Let $R=\frac{\mathbb{F}_p[X_1,X_2,\ldots , X_{d+1}]}{(X_1^{p+1}+X_2^{p+1}+\cdots +X_{d+1}^{p+1})},$ $I=(x_1,x_2,\ldots ,x_d)$ and $\m=(I, x_{d+1}).$ Then\\
\rm {(1)}  $(I^n)^*=I^n+\m ^{n+1}.$\\
\rm {(2)} $P_I^*(n)=(p+1)\binom{n+d-1}{d}-\binom{p}{2}\binom{n+d-2}{d-1}+\cdots +(-1)^j\binom{p}{j+1}\binom{n+d-j-1}{d-j},$ where $j=\min \{p-1,d\}.$\\ Moreover, $e_i^*(I)=0$ for $p\leq i \leq d$ and $\gr ^*_I(R)$ is Cohen-Macaulay.
\end{Theorem}
\begin{proof}
\rm{(1)} and \rm{(2)} are immediate consequences of Theorem \ref{p+1 deg ch p} and Proposition \ref{HP_t} respectively.
\end{proof}
\section{Quartic diagonal hypersurfaces} \label{sec_QuarticTHP}
Let $R=\frac{\mathbb{F}_p[X,Y,Z]}{(X^r+Y^r+Z^r)}$ and $I=(y,z).$ Let $p$ be a prime number such that $p \nmid r.$ Let $\mathcal{T} =\{(I^n)^*\}_{n\geq 0}$ be tight Hilbert filtration. Since $I$ is a reduction of $\m$ and $\gr_{\m}(R)$ is a domain, we have $\overline{I^n}=\overline{\m ^n}=\m ^n.$ Since $I$ is generated by a regular sequence, from Huneke-Itoh intersection theorem \cite{H1987,Itoh88}, it follows that for all $n\geq 0,$
$$I^n \cap \overline{I^{n+1}}=I^n\cap \overline{\m ^{n+1}}=\m I^n.$$
Using Theorem \ref{SV_Hyp} it follows that for $p>r-2,$ we have $I^*=I+\m ^2.$ By Theorem \ref{gentytclosure}, it follows that $(I^n)^*=I^n+\m ^{n+1}.$
Using Proposition \ref{HP_t}, we get
$$P_I^*(n) =r\binom {n+1}{2}-\binom {r-1}{2}n+\binom {r-1}{3}.$$
When $r=3,$ $$P_I^*(n)= 3\binom {n+1}{2}-n.$$\\
When $r=4,$ $$P_I^*(n)= 4\binom {n+1}{2}-3n+1.$$
These agree with the formulas derived in \cite{GMV2019}.
We discuss about tight closure of parameter ideals in $$R=\frac{\mathbb{F}_p[X,Y,Z,W]}{(X^4+Y^4+Z^4+W^4)}.$$ The Hilbert series and the Hilbert polynomial of $R$ are given by
\begin{align*}
H(R,\lambda)&=\sum_{n=0}^\infty \dim R_n \lambda ^n=\frac{1-\lambda^4}{(1-\lambda)^4},\\
P_R(n)&=e_0\binom{n+2}{2}-e_1(\m)(n+1)+e_2(\m)=4\binom{n+2}{2}-6(n+1)+4.
\end{align*}
Since, $\deg H(R,\lambda)=0,$ $H_R(0)\neq P_R(0)\text{ and }H_R(n)=P_R(n) \text{ for all }n\geq 1.$ 
\begin{Theorem}
Let $R=\frac{ \mathbb{F}_p [X,Y,Z,W]}{(X^4+Y^4+Z^4+W^4)},$ $I=(x,y,z)$ and $\m=(I,w).$ Then
\[
    I^*= 
\begin{cases}
    \m,& \text{ if } p=2,\\
    I+\m ^2,& \text{ if } p=3,\\
    I+\m ^3,& \text{ if } p\geq 5.
\end{cases}
\]
\end{Theorem}
\begin{proof}
By the Brian\c{c}on-Skoda Theorem, we have, $I+\mathfrak{m}^3\subseteq I^*.$ We check if $w,w^2\in I^*.$ {We note that $x^3,y^3,z^3,w^3$ are test elements of $R$ by Theorem \ref{Thm:Hoc_testele}.}
We find $I^*$ for various primes.\\
\textbf{Case (i)} Suppose $p=2$. Clearly, $w^4=-x^4-y^4-z^4.$
We have $$w^{2^e}=(x^4+y^4+z^4)^{2^{e-2}}.$$\\
For $q=2^e,$ $w^q=x^q+y^q+z^q\in I^{[q]}.$ Hence $w\in I^*.$ This implies that $I^*=\mathfrak{m}.$\\
\textbf{Case (ii)} Suppose $p=3$. Clearly, 
\begin{equation}
(w^4)^{3^e}=-(x^4)^{3^e}-(y^4)^{3^e}-(z^4)^{3^e}\in I^{[3^{e+1}]}
\end{equation}
Hence $(w^2)^{3^e}(w^2)^{3^e}(w^2)^{3^e}=(w^2)^{3^{e+1}}\in I^{[3^{e+1}]}$ for all $e.$
Therefore $w^2\in I^*.$ 
We have $w^3w^{27}\notin I^{[27]}.$ If possible let $w^{30}\in I^{[27]}.$ Then 
\begin{equation}\label{eq:foriniideal}
W^{30}\in (X^{27},Y^{27},Z^{27})+(X^4+Y^4+Z^4+W^4)
\end{equation}
in $\mathbb{F}_p[X,Y,Z,W].$
Let $>$ denote the graded reverse lex ordering with $W>X>Y>Z$ in $\mathbb{F}_p[X,Y,Z,W].$ Note that Lemma \ref{comp_initialideal} yields a contradiction to equation \ref{eq:foriniideal}. Using Theorem \ref{Thm:Hoc_testele}, $w^3$ is a test element, which implies that $w\notin I^*.$Therefore, $I^*=(x,y,z,w^2)=I+\mathfrak{m}^2.$ \\
\textbf{Case (iii)} Suppose $p=5$. Since $ p \equiv 1 \pmod {4}$ by \cite[Proposition 5.21 (c)]{HR}, it follows that $R$ is F-pure. Hence by Proposition  \ref{Prop_FedWat_test}, the test ideal is a radical ideal of $R$ and since $x^3,y^3,z^3,w^3\in \tau(R),$ it follows that $x,y,z,w\in \tau(R).$ We will show that $I^*=I+\mathfrak{m}^3.$ If possible let $w^2\in I^*$ then $ww^2\in I$ as $w$ is a test element. This implies that $$W^3\in (X,Y,Z)+(X^4+Y^4+Z^4+W^4).$$ Consider the graded reverse lex ordering $>$ with $W>X>Y>Z$ in $\mathbb{F}_p[X,Y,Z,W]$ and use Lemma \ref{comp_initialideal} to obtain a contradiction. Hence it follows that $w^2 \notin I^*.$ Hence $I^*=(x,y,z,w^3).$\\
\textbf{Case (iv)} Suppose $p>5$. It follows from Theorem \ref{SV_Hyp} that,
$$I^*=I+\mathfrak{m}^3.$$ Hence $R$ is not $F$-rational.
\end{proof}
In order to compute the tight Hilbert polynomial, we find $(I^n)^*.$
\begin{Proposition}\label{quartictytclofpow}
Let $R=\frac{\mathbb{F}_p[X,Y,Z,W]}{(X^4+Y^4+Z^4+W^4)},$ $I=(x,y,z)$ and $\m=(I,w).$ Then for $n\geq 1,$ \\
\[
    (I^n)^*= 
\begin{cases}
    I^n+(x+y+z+w),&  \text{ if }p=2,\\
    I^n+\m ^{n+1},& \text{ if } p=3,\\
    I^n+\m ^{n+2}, & \text{ if } p\geq 5.
\end{cases}
\]
\end{Proposition}
\begin{proof}
\textbf{Case (i)} Suppose $p=2.$ It is clear that the only minimal prime of $R$ is $(x+y+z+w)=\mathfrak{p},$ say. Since $R/\mathfrak{p}$ is a polynomial ring, every ideal in $R/\mathfrak{p}$ is tightly closed. We know that for any ideal $J$ in $R$ and $r\in R,$ $r\in J^*$ if and only if $\overline{r}\in \left( J\frac{R}{\mathfrak{p}} \right)^*$ for $\overline{r}\in R/\mathfrak{p}.$\\ Since $$\left( \frac{I^n+\mathfrak{p}}{\mathfrak{p}}\right)^*=\frac{I^n+(x+y+z+w)}{(x+y+z+w)}=\frac{I^n+\mathfrak{p}}{\mathfrak{p}} \text{ in } \frac{R}{\mathfrak{p}},$$ it implies that $I^*=I+(x+y+z+w)=\m$ and $(I^n)^*=I^n+(x+y+z+w).$ Note that $\{(I^n)^*\}_{n\geq 0}$ is not an $I$-admissible filtration.\\
\textbf{Case (ii)} Suppose $p=3.$ From Theorem \ref{THP ch p deg p+1} \rm {(1)}, it follows that $(I^n)^*=I^n+\m ^{n+1}.$\\
\textbf{Case (iii)} Suppose $p=5.$ If $ p \equiv 1 (\mod 4),$ then $R$ is F-pure. It follows from \cite[Lemma 3.1]{AHS} that $$(I^n)^*=I^{n-1}(I+\mathfrak{m}^3)=I^n+\m ^{n+2} \text{ as } r(\m)=3.$$
\textbf{Case (iv)} Suppose $p> 5.$ We use Theorem \ref{gentytclosure} to find out the tight closure of powers of ideal $I$ in $R.$
We show that $pd(R/I^n)<\infty.$ The Koszul complex gives a projective resolution in the case when $n=1.$ By using induction on $n$ and the following short exact sequence one can easily conclude that the projective dimension of $R/I$ and $R/I^n$ are equal for all $n.$
$$0\longrightarrow I^{n-1}/I^n \longrightarrow R/I^n\longrightarrow R/I^{n-1}\longrightarrow 0$$
The Eagon-Northcott complex for $R/I^n$ gives a minimal resolution which is linear. The degree of all generators of $I^n$ is $n.$
When $p> 5,$ $R$ satisfies Strong Vanishing Conjecture by Theorem \ref{SV_Hyp}. In this case $b_{3i}=n+2$ (refer Theorem \ref{gentytclosure}). Therefore it follows that $(I^n)^*=I^n+\mathfrak{m}^{n+2}.$
\end{proof}
\begin{Theorem}
Let $R=\frac{ \mathbb{F}_p [X,Y,Z,W]}{(X^4+Y^4+Z^4+W^4)},$ $I=(x,y,z)$ and $\m=(I,w).$ Then for $n\geq 1,$
\[
    P_I^*(n)= 
\begin{cases}
    4\binom{n+2}{3}-3\binom{n+1}{2}+n,& \text{ if } p=3,\\
    4\binom{n+2}{3}-\binom{n+1}{2},& \text{ if } p\geq 5.
\end{cases}
\]
Moreover $\gr ^*_I(R)$ is Cohen-Macaulay.
\end{Theorem}
\begin{proof}
\textbf{Case (i)} Suppose $p=3$. Use Theorem \ref{THP ch p deg p+1} \rm {(2)}.\\
\textbf{Case (ii)} Suppose $p=5$. Since $p \equiv 1 \pmod {4},$ the result follows from Theorem \ref{THP_Fpure}.\\
\textbf{Case (iii)} Suppose $p>5$. It follows as a consequence of Theorem \ref{THP_SVT}.
\end{proof}

\bibliographystyle{plain}
\bibliography{ref}

\begin{thebibliography}{10}

\bibitem{Ab_94}
Ian~M. Aberbach.
\newblock Tight closure in {$F$}-rational rings.
\newblock {\em Nagoya Math. J.}, 135:43--54, 1994.

\bibitem{AHS}
Ian~M. Aberbach, Craig Huneke, and Karen~E. Smith.
\newblock A tight closure approach to arithmetic {M}acaulayfication.
\newblock {\em Illinois J. Math.}, 40(2):310--329, 1996.

\bibitem{BH}
Winfried Bruns and J\"{u}rgen Herzog.
\newblock {\em Cohen-{M}acaulay rings}, volume~39 of {\em Cambridge Studies in
  Advanced Mathematics}.
\newblock Cambridge University Press, Cambridge, 1993.

\bibitem{BrunsVetter88}
Winfried Bruns and Udo Vetter.
\newblock {\em Determinantal rings}, volume~45 of {\em Monograf\'{\i}as de
  Matem\'{a}tica [Mathematical Monographs]}.
\newblock Instituto de Matem\'{a}tica Pura e Aplicada (IMPA), Rio de Janeiro,
  1988.

\bibitem{EisenbudCommAlg}
David Eisenbud.
\newblock {\em Commutative algebra}, volume 150 of {\em Graduate Texts in
  Mathematics}.
\newblock Springer-Verlag, New York, 1995.
\newblock With a view toward algebraic geometry.

\bibitem{EisenbudGeomSyzygies}
David Eisenbud.
\newblock {\em The geometry of syzygies}, volume 229 of {\em Graduate Texts in
  Mathematics}.
\newblock Springer-Verlag, New York, 2005.
\newblock A second course in commutative algebra and algebraic geometry.

\bibitem{Fed_Wat}
Richard Fedder and Keiichi Watanabe.
\newblock A characterization of {$F$}-regularity in terms of {$F$}-purity.
\newblock In {\em Commutative algebra ({B}erkeley, {CA}, 1987)}, volume~15 of
  {\em Math. Sci. Res. Inst. Publ.}, pages 227--245. Springer, New York, 1989.

\bibitem{GMV2019}
Kriti Goel, J.~K. Verma, and Vivek Mukundan.
\newblock Tight closure of powers of ideals and tight {H}ilbert polynomials.
\newblock {\em Math. Proc. Cambridge Philos. Soc.}, 169(2):335--355, 2020.

\bibitem{Hoc_2004}
Melvin Hochster.
\newblock Tight closure theory and characteristic {$p$} methods.
\newblock In {\em Trends in commutative algebra}, volume~51 of {\em Math. Sci.
  Res. Inst. Publ.}, pages 181--210. Cambridge Univ. Press, Cambridge, 2004.
\newblock With an appendix by Graham J. Leuschke.

\bibitem{Hoc_Hun_1994}
Melvin Hochster and Craig Huneke.
\newblock {$F$}-regularity, test elements, and smooth base change.
\newblock {\em Trans. Amer. Math. Soc.}, 346(1):1--62, 1994.

\bibitem{HR}
Melvin Hochster and Joel~L. Roberts.
\newblock The purity of the {F}robenius and local cohomology.
\newblock {\em Advances in Math.}, 21(2):117--172, 1976.

\bibitem{HM}
Sam Huckaba and Thomas Marley.
\newblock Hilbert coefficients and the depths of associated graded rings.
\newblock {\em J. London Math. Soc. (2)}, 56(1):64--76, 1997.

\bibitem{H1987}
Craig Huneke.
\newblock Hilbert functions and symbolic powers.
\newblock {\em Michigan Math. J.}, 34(2):293--318, 1987.

\bibitem{Huneke1998}
Craig Huneke.
\newblock Tight closure, parameter ideals, and geometry.
\newblock In {\em Six lectures on commutative algebra ({B}ellaterra, 1996)},
  volume 166 of {\em Progr. Math.}, pages 187--239. Birkh\"{a}user, Basel,
  1998.

\bibitem{kod_van}
Craig Huneke and Karen~E. Smith.
\newblock Tight closure and the {K}odaira vanishing theorem.
\newblock {\em J. Reine Angew. Math.}, 484:127--152, 1997.

\bibitem{Itoh88}
Shiroh Itoh.
\newblock Integral closures of ideals generated by regular sequences.
\newblock {\em J. Algebra}, 117(2):390--401, 1988.

\bibitem{marthesis}
Thomas~John Marley.
\newblock {\em Hilbert functions of ideals in {C}ohen-{M}acaulay rings}.
\newblock ProQuest LLC, Ann Arbor, MI, 1989.
\newblock Thesis (Ph.D.)--Purdue University.

\bibitem{reesau}
D.~Rees.
\newblock A note on analytically unramified local rings.
\newblock {\em J. London Math. Soc.}, 36:24--28, 1961.

\end{thebibliography}
\end{document}